\documentclass[12pt]{article}
\usepackage{amsmath, amsthm}
\usepackage{enumerate}
\usepackage{amssymb}
\usepackage{geometry}
\newtheorem{thm}{Theorem}[section]
\newtheorem{exam}{Example}[section]
\newtheorem{cor}[thm]{Corollary}
\newtheorem{lem}[thm]{Lemma}
\newtheorem{ques}[thm]{Question}
\newtheorem{prop}[thm]{Proposition}

\theoremstyle{definition}
\newtheorem{defn}{Definition}[section]
\theoremstyle{remark}
\newtheorem{rem}{Remark}[section]

\DeclareMathOperator{\ptc}{\xrightarrow[]{p_\tau}}
\DeclareMathOperator{\uptc}{\xrightarrow[]{up_\tau}}
\DeclareMathOperator{\tc}{\xrightarrow[]{\tau}}
\DeclareMathOperator{\tcc}{\xrightarrow[]{\acute{\tau}}}

\DeclareMathOperator{\wc}{\xrightarrow[]{w}}

\DeclareMathOperator{\oc}{\xrightarrow[]{o}}

\newcommand{\eval}[2][\right]{\relax
  \ifx#1\right\relax \left.\fi#2#1\rvert}

\begin{document}

\title{\bf Compact-Like Operators in Vector Lattices Normed by Locally Solid Lattices} 
\maketitle

\author{\centering Abdullah AYDIN \\ \bigskip  \small  
	Department of Mathematics, Mu\c{s} Alparslan University, Mu\c{s}, Turkey. \\}

\bigskip

\abstract{A linear operator $T$ between two vector lattices normed by locally solid Riesz spaces is said to be $p_\tau$-continuous if, for any $p_\tau$-null net $(x_\alpha)$, the net $(Tx_\alpha)$ is $p_\tau$-null, and $T$ is said to be $p_\tau$-bounded operator if it sends $p_\tau$-bounded subsets to $p_\tau$-bounded subsets. Also, $T$ is called $p_\tau$-compact if, for any $p_\tau$-bounded net $(x_\alpha)$, the net $(Tx_\alpha)$ has a $p_\tau$-convergent subnet. They generalize several known classes of operators such as norm continuous, order continuous, $p$-continuous, order bounded, $p$-bounded, compact and AM-compact operators. We study the general properties of these operators.}

\bigskip
\let\thefootnote\relax\footnotetext
{
	
\text{2010 AMS Mathematics Subject Classification:} 46A40, 	54A20, 46A03

e-mail: aaydin.aabdullah@gmail.com}

\section{Introduction}
Compact operators provide natural and effective tools in functional analysis. In the present paper, the aim is to introduce and study compact-like operators in vector lattices normed by locally solid vector lattices. Recently, many papers are devoted to the concept of unbounded convergence; see for example \cite{A,AA,AGG,AEEM,AEEM2,DOT,GTX,KMT}. It is well-investigated in vector lattices and locally solid vector lattices \cite{AB,ABPO,K,V}. We refer to the reader for detail information about the operator theory, the theory of locally solid vector lattice, and lattice-normed vector lattice; see \cite{AB,ABPO,AA,AEEM,AEEM2,L,M}.

Recall that a net $(x_\alpha)_{\alpha\in A}$ in a vector lattice $X$ is {\em order convergent} to $x\in X$ if there exists another net $(y_\beta)_{\beta\in B}$ satisfying $y_\beta \downarrow 0$, and, for any $\beta\in B$, there exists $\alpha_\beta\in A$ such that $\lvert x_\alpha-x\rvert\leq y_\beta$ for all $\alpha\geq\alpha_\beta$. In this case, we write $x_\alpha\oc x$. In a vector lattice $X$, a net $(x_\alpha)$ is {\em unbounded order convergent} to $x\in X$ if $\lvert x_\alpha-x\rvert\wedge u\oc 0$ for every $u\in X_+$. Let $X$ be a vector space, $E$ be a vector lattice, and $p:X\to E_+$ be a vector norm (i.e. $p(x)=0\Leftrightarrow x=0$, $p(\lambda x)=|\lambda|p(x)$ for all $\lambda\in\mathbb{R}$, $x\in X$, and $p(x+y)\leq p(x)+p(y)$ for all $x,y\in X$) then the triple $(X,p,E)$ is called a {\em lattice-normed space}, abbreviated as $LNS$. A linear operator $T$ between two $LNS$s $(X,p,E)$ and $(Y,m,F)$ is said to be \textit{dominated} if there is a positive operator $S:E\to F$ satisfying $m(Tx)\leq S(p(x))$ for all $x\in X$. In an $LNS$ $(X,p,E)$ a subset $A$ of $X$ is called {\em $p$-bounded} if there exists $e\in E$ such that $p(a)\leq e$ for all $a\in A$; see \cite[Def.2]{AEEM}. The \textit{mixed-norm} on an $LNS$ $(X,p,E)$ is defined by $p\text{-}\lVert x\rVert_E=\lVert p(x)\rVert_E$ for all $x\in X$. We refer the reader for more information on $LNS$s to \cite{BGKKKM,K} and \cite{AEEM}. If $X$ is a vector lattice and the vector norm $p$ is monotone (i.e. $|x|\leq |y|\Rightarrow p(x)\leq p(y)$) then the triple $(X,p,E)$ is called a {\em lattice-normed vector lattice}, abbreviated as $LNVL$; see \cite{AGG,AEEM,AEEM2}.  

A subset $A$ of vector lattice is called {\em solid} whenever $\lvert x\rvert\leq\lvert y\rvert$ and $y\in A$ imply $x\in A$. Let $E$ be a vector lattice and $\tau$ be a linear topology on $E$ that has a base at zero consisting of solid sets. Then the pair $(E,\tau)$ is said a {\em locally solid vector lattice} (or, {\em locally solid lattice}, or {\em locally solid Riesz space}). A locally solid lattice $(E,\tau)$ is said to have the {\em Lebesgue property} if, for any net $(x_\alpha)$ in $E$, $x_\alpha\oc 0$ implies $x_\alpha\tc 0$, and is also said to satisfy the {\em Fatou property} if $\tau$ has a base at zero consisting of solid and order closed sets. It follows from \cite[Thm.2.28]{AB} that a linear topology $\tau$ on a vector lattice $E$ is a locally solid iff it is generated by a family of Riesz pseudonorms $\{\rho_j\}_{j\in J}$. Moreover, if a family of Riesz pseudonorms generates a locally solid topology $\tau$ on a vector lattice $E$ then $x_\alpha \tc x$ iff $\rho_j(x_\alpha-x)\to 0$ in $\mathbb{R}$ for each $j\in J$. In this article, unless otherwise, the pair $(E,\tau)$ refers to as a locally solid lattice, and the topologies in locally solid lattices are generated by families of Riesz pseudonorms $\{\rho_j\}_{j\in J}$. Also, all vector lattices are assumed to be real and Archimedean.

Let $(X,p,E)$ be an $LNVL$ with $(E,\tau)$ being a locally solid lattice. Then $(X,p,E_\tau)$ is said to be a {\em vector lattice normed by locally solid Riesz space} (or, {\em vector lattice normed by locally solid lattice}), abbreviated as $LSNVL$ in \cite{AA}. Throughout this article, we use $X$ instead of $(X,p,E_\tau)$, and $Y$ instead of $(Y,m,F_{\acute{\tau}})$. Note that $L(X,Y)$ denotes the space of all linear operators between vector spaces $X$ and $Y$. If $X$ is a normed space then $X^*$ denotes the topological dual of $X$. We abbreviate the convergence $p(x_{\alpha}-x)\tc 0$ as $x_\alpha\ptc x$, and say in this case that $(x_\alpha)$ $p_\tau$-converges to $x$. A net $(x_\alpha)_{\alpha \in A}$ in an $LSNVL$ $(X,p,E_\tau)$ is said to be {\em $p_\tau$-Cauchy} if the net $(x_\alpha-x_{\alpha'})_{(\alpha,\alpha') \in A\times A}$ $p_\tau$-converges to $0$. An $LSNVL$ $(X,p,E_\tau)$ is called (\textit{sequentially}) {\em $p_\tau$-complete} if every $p_\tau$-Cauchy (sequence) net in $X$ is $p_\tau$-convergent. In an $LSNVL$ $(X,p,E_\tau)$, a subset $A$ of $X$ is called {\em $p_\tau$-bounded} if $p(A)$ is $\tau$-bounded in $E$. An $LSNVL$ $(X,p,E_\tau)$ is called {\em $op_\tau$-continuous} if $x_\alpha\oc 0$ in $X$ implies $p(x_\alpha)\tc 0$ in $E$. A net $(x_\alpha)$ in an $LSNVL$ $(X,p,E_\tau)$ is said to be {\em unbounded $p_\tau$-convergent} to $x\in X$ (shortly, $(x_\alpha)$ $up_\tau$-converges to $x$ or $x_\alpha\uptc x$) if $p(\lvert x_\alpha-x\rvert\wedge u)\tc 0$ for all $u\in X_+$; see \cite{AA}. In this paper, we frequently use the following lemma and so we shall keep in mind it; see \cite[Lem.1.1]{AA}.

\begin{lem}\label{monoton is convergent}
If $(x_\alpha)_{\alpha\in A}$ and $(y_\alpha)_{\alpha\in A}$ be two nets in a locally solid vector lattice $(E,\tau)$ such that $\lvert x_\alpha\rvert\leq \lvert y_\alpha\rvert$ for all $\alpha\in A$ and $y_\alpha\tc 0$ then $x_\alpha\tc 0$. 
\end{lem}
%%%%%%%%%%%%%%%%%%%%%%%%%%%%%%%%%%%%%%%%%%%%%%%%%%%%%%%%%%%%
\section{$p_\tau$-Continuous and $p_\tau$-Bounded Operators}

In this section, we give the notions of $p_\tau$-continuous and $p_\tau$-bounded operators. 
\begin{defn}
Let $X$, $Y$ be two $LSNVL$s and $T\in L(X,Y)$. Then
\begin{enumerate}
\item[(1)] $T$ is called {\em $p_\tau$-continuous} if $x_\alpha\ptc 0$ in $X$  implies $Tx_\alpha\ptc 0$ in $Y$, and if the condition holds only for sequences then $T$ is called {\em sequentially $p_\tau$-continuous},

\item[(2)] $T$ is called {\em $p_\tau$-bounded} if it maps $p_\tau$-bounded sets in $X$ to $p_\tau$-bounded sets in $Y$.
\end{enumerate}
\end{defn}

\begin{rem}\
\begin{enumerate}
\item [(i)]	Let $T,S:(X,p,E_\tau)\to (Y,m,F_{\acute{\tau}})$ be $p_\tau$-continuous operators. Then $\lambda S+\mu T$ is also $p_\tau$-continuous for any real numbers $\lambda$ and $\mu$. In particular, if $H=T-S$ then $H$ is $p_\tau$-continuous. Moreover, if $-T_1\leq T \leq T_2$ with $T_1\ and \ T_2$ are positive and $p_\tau$-continuous operators then $T$ is $p_\tau$-continuous.
	
\item[(ii)] Let $T$ be an operator between $LSNVL$s $(X,p,E_\tau)$ and $(Y,m,F_{\acute{\tau}})$ with $(E,\tau)$ and $(F,\acute{\tau})$ having order bounded neighborhoods of zero. Then, by applying \cite[Thm.2.19(i)]{AB} and \cite[Thm.2.2]{L}, one can see that $T$ is $p$-bounded iff it is $p_\tau$-bounded. Moreover, $T:(E,\lvert \cdot \rvert,E_\tau)\to(F,\lvert \cdot \rvert,F_{\acute{\tau}})$ is $p_\tau$-bounded iff $T:X\to Y$ is order bounded.

\item[(iii)] Let $X$ be a vector lattice and $(Y,\lVert\cdot\rVert_Y)$ be a normed space. Then $T\in L(X,Y)$ is called \textit{order-to-norm continuous} if $x_\alpha\oc 0$ in $X$ implies $Tx_\alpha\xrightarrow{\lVert\cdot\rVert_Y} 0$; see \cite[Sect.4,p.468]{MMP}. For a locally solid lattice $(X,\tau)$ with the Lebesgue property, the $p_\tau$-continuity of $T:(X,\lvert\cdot\rvert,X_\tau)\to(Y,\lVert\cdot\rVert_Y,\mathbb{R})$ implies order-to-norm continuity of it.

\item[(iii)] Let $X$ be a vector lattice and $(Y,m,F_{\acute{\tau}})$ be an $LSNVL$, and $T:X\to Y$ be a strictly positive operator. Define $p:X\to F_+$, by $p(x)=(m\circ T)(\lvert x\rvert)$. Then $(X,p,F_{\acute{\tau}})$ is an $LSNVL$, and also, the map $T:(X,p,F_{\acute{\tau}})\to (Y,m,F_{\acute{\tau}})$ is $p_\tau$-continuous.
\end{enumerate}
\end{rem}

\begin{prop}\label{dominant opt is bounded}
Any dominated operator $T$ from an $LSNVL$ $(X,p,E_\tau)$ with $(E,\tau)$ has an order bounded $\tau$-neighborhood of zero to an $LSNVL$ $(Y,m,F_{\acute{\tau}})$ is $p_\tau$-bounded.
\end{prop}

\begin{proof}
Consider a $p_\tau$-bounded subset $A$ in $X$. That is, $p(A)$ is $\tau$-bounded in $E$. So, $p(A)$ is order bounded in $E$; see \cite[Thm.2.2]{L}. Let $S$ be dominant of $T$. Since $S$ is positive operator, $S\big(p(A)\big)$ is order bounded in $F$. Also, we know that $m\big(T(a)\big)\leq S\big(p(a)\big)$ for all $a\in A$, and so $m\big(T(A)\big)$ is order bounded in $F$. Hence, by applying \cite[Thm.2.19(i)]{AB}, $m\big(T(A)\big)$ is $\acute{\tau}$-bounded in $F$. Therefore, $T$ is $p_\tau$-bounded.
\end{proof}

The converse of Proposition \ref{dominant opt is bounded} is not true in the general. For instance, consider $\ell_\infty$ with the norm topology and $\mathbb{R}$ with the usual topology, and the identity operator $I:(\ell_\infty,\lvert \cdot \rvert,\ell_\infty)\to(\ell_\infty,\lVert\cdot\rVert,\mathbb{R})$. It is $p_\tau$-bounded. Indeed, for any $p_\tau$-bounded set $A$ in $\ell_\infty$, $\lvert A\rvert$ is $\tau$-bounded in $\ell_\infty$. Thus $\big\lVert \lvert A\lvert \big\rVert=\lVert A\rVert$ is bounded in $\mathbb{R}$. But it is not dominated; see \cite[Rem.p.388]{BGKKKM}. Next proposition gives a relation between the $p_\tau$- and order continuity.
\begin{prop}\label{Toc}
Let $(Y,m,F_{\acute{\tau}})$ be arbitrary and $(X,p,E_\tau)$ be $op_\tau$-continuous $LSNVL$s, and $T:(X,p,E_\tau)\to(Y,m,F_{\acute{\tau}})$ be a $($sequentially$)$ $p_\tau$-continuous positive operator. 
Then $T:X\to Y$ is $($$\sigma$-$)$ order continuous operator.
\end{prop}

\begin{proof}
Assume $x_\alpha\downarrow 0$ in $X$. Since $X$ is $op_\tau$-continuous, we have $p(x_\alpha)\tc 0$, and so $x_\alpha\ptc 0$ in $X$. By the $p_\tau$-continuity of $T$, $m(Tx_\alpha)\tcc 0$ in $F$. It can be seen that $Tx_\alpha\downarrow$ because $T$ is positive. Then, applying \cite[Prop.2.4]{AA}, we get $Tx_\alpha\downarrow 0$. Thus, $T$ is order continuous.
\end{proof}

\begin{cor}
Let $(X,p,E)$ be an $op_\tau$-continuous $LSNVL$ and $(Y,m,F_{\acute{\tau}})$ be an $LSNVL$ with $Y$ being order complete. If $T:(X,p,E)\to(Y,m,F)$ is $p_\tau$-continuous and $T\in L^\sim (X,Y)$ then $T:X\to Y$ is order continuous.
\end{cor}

\begin{proof}
Since  $Y$ is order complete and  $T$ is order bounded, by Riesz-Kantorovich formula, we have $T=T^+-T^-$. 
Now, Proposition \ref{Toc} implies that $T^+$ and $T^-$ are both order continuous, and so $T$ is order continuous.
\end{proof}

The following work, which is $p_\tau$-version of \cite[Prop.3]{AEEM2}, gives norm continuity of sequentially $p_\tau$-continuous operator on the mixed-norms.
\begin{prop}\label{seqpcontimpliesnormcont}
Let $(X,p,E_\tau)$ and $(Y,m,F_{\acute{\tau}})$ be two $LSNVL$s with $(E,\lVert\cdot\rVert_E)$ and $(F,\lVert \cdot \rVert_F)$ being normed vector lattices, and where $\tau$ and $\acute{\tau}$ are generated by the norms. If $T:(X,p,E_\tau)\to(Y,m,F_{\acute{\tau}})$ is sequentially $p_\tau$-continuous then $T:(X,p\text{-}\lVert\cdot\rVert_E)$ $\to(Y,m\text{-}\lVert\cdot\rVert_F)$ is norm continuous.
\end{prop}

\begin{rem}\label{ptcontinuous not order bounded}
By applying \cite[Thm.2.19(i)]{AB} and \cite[Prop.4]{AEEM2}, one can see that every $p_\tau$-continuous operator is $p_\tau$-bounded. But, a $p_\tau$-continuous operator $T:(X,p,E_\tau)\to(Y,m,F_{\acute{\tau}})$ need not to be order bounded from $X$ to $Y$. Indeed, consider the classical “Fourier coefficients” operator $T:L_1[0,1]\to c_0$ defined by the formula
$$
T(f)=\bigg ( \int_{0}^{1} f(x)sinx \ dx,\int_{0}^{1} f(x)sin2x\ dx,... \bigg).
$$
Then $T:L_1[0,1]\to c_0$ is norm bounded, but it  is not order bounded; see \cite[Exer.10,p.289]{AB}. So, $T: (L_1[0,1],\lVert \cdot \rVert_{L_1},\mathbb{R})\to (c_0,\lVert \cdot \rVert_{\infty},\mathbb{R})$ is $p_\tau$-continuous and is not order bounded.
\end{rem}

Using \cite[Thm.2.2]{L} in Remark \ref{ptcontinuous not order bounded}, it can be seen that $p_\tau$-continuity implies orded boundedness if $(F,\acute{\tau})$ has order bounded $\acute{\tau}$-neighborhood of zero. Recall that an operator $T\in L(X,Y)$, where $X$ and $Y$ are normed spaces, is called \textit{Dunford-Pettis} if $x_n\wc 0$ in $X$ implies $Tx_n\xrightarrow{\lVert\cdot\rVert} 0$ in $Y$. The following is $p_\tau$-version of \cite[Prop.5]{AEEM2}, so we omit its proof.
\begin{prop}\label{dunford-pettis}
Let $(X,\lVert\cdot\rVert_X)$ be a normed vector lattice and $(Y,\lVert\cdot\rVert_Y)$ be a normed space. 
Put $E:=\mathbb{R}^{X^*}$  and define $p:X\to E_+$, by $p(x)[f]=\lvert f\rvert(\lvert x\rvert)$ for $f\in X^*$. 
It is easy to see that $(X,p,E_\tau)$, where $\tau$ is the topology generated by the norm $\lVert \cdot \rVert_{X^*}$, is an $LSNVL$. Then the followings hold;
\begin{enumerate}
\item[(i)] If $T\in L(X,Y)$ is a Dunford-Pettis operator then $T:(X,p,E_\tau)\to(Y,\lVert\cdot\rVert_Y,\mathbb{R})$ is sequentially $p_\tau$-continuous.
		
\item[(ii)] The converse holds if the lattice operations of $X$ are weakly sequentially continuous.
\end{enumerate}
\end{prop}

%%%%%%%%%%%%%%%%%%%%%%%%%%%%%%%%%%%%%%%%%%%%%%%%%%%%%%%%%%%%%%%%%%

%%%%%%%%%%%%%%%%%%%%%%%%%%%%%%%%%%%%%%%%%%%%%%%%%%%%%%%%%%%%%%%%
\section{$up_\tau$-Continuous Operators}
Recall that a net $(x_\alpha)$ in an $LSNVL$ $(X,p,E_\tau)$ is said to be {\em unbounded $p_\tau$-convergent} to $x$ if $p(\lvert x_\alpha-x\rvert \wedge u)\xrightarrow{\tau}0$ for all $u\in X_+$; see \cite{AA}. 
\begin{defn}
An operator $T$ between two $LSNVL$s $X$ and $Y$ is called {\em $up_\tau$-continuous} if it maps the $up_\tau$-convergent net to  $up_\tau$-convergent nets. If it holds only for sequence then $T$ is called {\em sequentially $up_\tau$-continuous}.
\end{defn}

It is clear that if $T$ is $($sequentially$)$ $p_\tau$-continuous operator then $T$ is $($sequentially$)$ $up_\tau$-continuous. For an $LSNVL$ $(X,p,E_\tau)$, a sublattice $Y$ of $X$ is called {\em $up_\tau$-regular} if, for any net $(y_\alpha)$ in $Y$, the convergence $y_\alpha\uptc 0$ in $Y$ implies $y_\alpha\uptc 0$ in $X$. The following is a more general extension of \cite[Prop.9.4.]{KMT}.

\begin{thm}\label{sequentially $up$-continuous}
Let $(X,p,E_\tau)$ and $(Y,m,F_{\acute{\tau}})$ be $LSNVL$s with $(E,\lVert\cdot\rVert_E)$ being a Banach lattice and $(F,\lVert\cdot\rVert_F)$ being normed vector lattice, and also $\tau$ and $\acute{\tau}$ are being generated by the norms. Then the followings hold;
\begin{enumerate}
\item[(i)] A dominated surjective lattice homomorphism operator $T\in L(X,Y)$ is sequentially $up_\tau$-continuous;

\item[(ii)] If $T\in L(X,Y)$ is a dominated lattice homomorphism operator and $T(X)$ is $up_\tau$-regular in $Y$ then it is sequentially $up_\tau$-continuous;

\item[(iii)] If $T\in L(X,Y)$ is a dominated lattice homomorphism operator and $I_{T(X)}$ $($the ideal generated by $T(X)$$)$ is $up_\tau$-regular in $Y$ then it is sequentially $up_\tau$-continuous.
\end{enumerate}
\end{thm}

\begin{proof}
$(i)$ Let's fix a net $x_n\uptc 0$ in $X$ and $u\in Y_+$. Since $T$ is a surjective lattice homomorphism, we have some $v\in X_+$ such that $Tv=u$. So, we have $p(\lvert x_n\rvert\wedge v)\tc 0$ in $E$. 
Since $T$ is dominated, there is a positive operator $S:E\to F$ such that 
$$
m\big(T(\lvert x_n\rvert\wedge v)\big)\leq S\big(p(\lvert x_n\rvert\wedge v)\big).
$$
Taking into account that $T$ is a lattice homomorphism and $Tv=u$, we get $m\big(\lvert Tx_n\rvert\wedge u\big)\leq S\big(p(\lvert x_n\rvert\wedge v)\big)$. By \cite[Thm.4.3]{AB}, we know that every positive operator from a Banach lattice to normed vector lattice is continuous, and so $S$ is continuous. Hence, we get $S\big(p(\lvert x_n\rvert\wedge v)\big)\tcc 0$ in $F$. That is, $m\big(\lvert Tx_n\rvert\wedge u\big) \tcc 0$, and we get the desired result.\\

$(ii)$ Since $T$ is a lattice homomorphism, $T(X)$ is vector sublattice of $Y$. So $\big(T(X),m,F_{\acute{\tau}}\big)$ is an $LSNVL$. Thus, by $(i)$, we have $T:(X,p,E_\tau)\to\big(T(X),m,F_{\acute{\tau}}\big)$ is sequentially $up_\tau$-continuous. 

Next, we show that $T:(X,p,E_\tau)\to(Y,m,F_{\acute{\tau}})$ is sequentially $up_\tau$-continuous. Consider an $up_\tau$-convergent to zero sequence $(x_n)$ in $X$. That is, $Tx_n\uptc 0$ in $T(X)$. Since $T(X)$ is $up_\tau$-regular in $Y$, $T(x_n)\uptc 0$ in $Y$. Therefore, $T$ is sequentially $up_\tau$-continuous.\\

$(iii)$ Let $(x_n)\uptc 0$ sequence in $X$. Thus, $p(\lvert x_n\rvert\wedge u)\tc 0$ in $E$ for all $u\in X_+$. Fix $0\leq w\in I_{T(X)}$. Then there is $x\in X_+$ such that $0\leq w\leq Tx$. For a dominant $S$, we have $	m\big(T(\lvert x_n\rvert\wedge x)\big)\leq S\big(p(\lvert x_n\rvert\wedge x)\big)$ and so, by taking lattice homomorphism of $T$, we have
$$ 
m\big((\lvert Tx_n\rvert\wedge Tx)\big)\leq S\big(p(\lvert x_n\rvert\wedge x)\big).
$$
It follows from $0\leq w\leq Tx$ that $m\big((\lvert Tx_n\rvert\wedge w)\big)\leq S\big(p(\lvert Tx_n\rvert\wedge x)\big)$. Now, the argument given in the proof of $(i)$ can be repeated here as well. Thus, we see that $T:(X,p,E_\tau)\to\big(I_{T(X)},m,F_{\acute{\tau}}\big)$ is sequentially $up_\tau$-continuous. Since $I_{T(X)}$ is $up_\tau$-regular in $Y$, it can be easily seen by $(ii)$ that $T:X\to Y$ is sequentially $up_\tau$-continuous.
\end{proof}

It should be mentioned, by using Theorem \ref{sequentially $up$-continuous}, that an operator, surjective lattice homomorphism with an order continuous dominant, is $up_\tau$-continuous.
\begin{prop}\label{S is uptau contınuous}
Let $(X,p,E_\tau)$ and $(Y,m,F_{\acute{\tau}})$ be two $LSNVL$s with $Y$ being order complete vector lattice. For a positive $up_\tau$-continuous operator $T:(X,p,E_\tau)\to (Y,m,F_{\acute{\tau}})$, consider the operator $S:(X_+,p,E_\tau)\to (Y_+,m,F_{\acute{\tau}})$ defined by $S(x)=\sup\{T(x_\alpha\wedge x):x_\alpha\in X_+,x_\alpha\uptc0\}$ for each $x\in X_+$. Then we have the followings;

\begin{enumerate}
\item[(i)] $S$ is $up_\tau$-continuous operator;
\item[(ii)] The Kantorovich extension of $S$ is $up_\tau$-continuous operator.
\end{enumerate}
\end{prop}

\begin{proof}
$(ii)$ We show firstly that $S$ has the Kantorovich extension. To make this let see additivity of it. By using \cite[Lem.1.4]{ABPO}, for any $up_\tau$-null net $(x_\alpha)$ in $X_+$, we have
$$
T\big((x+y)\wedge x_\alpha\big) \leq T(x\wedge x_\alpha)+T(y\wedge x_\alpha)\leq S(x)+S(y).
$$
So, by taking supremum, we get $S(x+y)\leq S(x)+S(y)$. On the other hand, for any two $up_\tau$-null nets $(x_\alpha)$ and $(y_\beta)$ in $X_+$, using the formula in the proof of \cite[Thm.1.28]{ABPO}, we get
$$
T(x\wedge x_\alpha)+T(y\wedge y_\beta)=T(x\wedge x_\alpha+y\wedge y_\beta)\leq T\big((x+y)\wedge (x_\alpha+y_\beta)\big)\leq S(x+y).
$$
So $S(x)+S(y)\leq S(x+y)$. By \cite[Thm.1.10]{ABPO}, $S$ extends to a positive operator, denoted by $\hat{S}:(X,p,E_\tau)\to (Y,m,F_{\acute{\tau}})$. That is $\hat{S}x=S(x^+)-S(x^-)$ for all $x\in X$. Now, we show $up_\tau$-continuity of $\hat{S}$. Fix a net $w_\beta\uptc 0$ in $X$. Then $w^+_\beta\uptc 0$ and $w^-_\beta\uptc 0$ in $X$, and so $S(w^+_\beta)\uptc 0$ and $S(w^-_\beta)\uptc 0$ in $Y$. Hence, $\hat{S}w_\beta=S(w^+_\beta)-S(w^-_\beta)\uptc 0$ in $Y$. 
\end{proof}

We complete this section wit the following technical work.
\begin{prop}\label{sup s is uptau cont}
Consider a positive $up_\tau$-continuous operator $T$  between $LSNVL$s $X$ and $Y$, and an ideal $A$ in $X$. Then an operator $S:(X,p,E_\tau)\to (Y,m,F_{\acute{\tau}})$ defined by $S(x)=\sup\limits_{a\in A}T(\lvert x\rvert\wedge a)$ for each $x\in X$ is $up_\tau$-continuous operator. 
\end{prop}

\begin{proof}
Let $x_\alpha\uptc 0$ be a net in $X$. Then $\lvert x_\alpha\rvert\uptc 0$, and so $T(\lvert x_\alpha\rvert)\uptc 0$ in $Y$. Thus, for each $u\in Y_+$, we have
$$
\lvert S(x_\alpha)\rvert\wedge u=\bigg\lvert\sup\limits_{a\in A}T(\lvert x_\alpha\rvert\wedge a)\bigg\rvert\wedge u\leq \big \lvert T(\lvert x_\alpha\rvert)\big\rvert\wedge u\leq T(\lvert x_\alpha\rvert)\wedge u\uptc 0.
$$
Therefore, $S(x_\alpha)\uptc 0$ in $Y$.
\end{proof}

%%%%%%%%%%%%%%%%%%%%%%%%%%%%%%%%%%%%%%%%%%%%%%%%%%%%%%%%%%%%%%
\section{The Compact-Like Operators}
In this section, we define the notions of $p_\tau$-compact and $up_\tau$-compact operators. 
\begin{defn}
Let $X$ and $Y$ be two $LSNVL$s and $T\in L(X,Y)$. Then $T$ is called {\em $p_\tau$-compact} if, for any $p_\tau$-bounded net $(x_\alpha)$ in $X$, there is a subnet $(x_{\alpha_\beta})$ such that $Tx_{\alpha_\beta}\ptc y$ in $Y$ for some $y\in Y$. If it holds only for sequence then $T$ is called {\em sequentially $p_\tau$-compact}.

\end{defn}

\begin{exam}\
\begin{enumerate}
\item[(i)] Let $(X,\lVert\cdot\rVert_X)$ and $(Y,\lVert\cdot\rVert_Y)$ be normed spaces. Then $T:(X,\lVert\cdot\rVert_X,\mathbb{R})\to(Y,\lVert\cdot\rVert_Y,\mathbb{R})$ is (sequentially) $p_\tau$-compact iff $T:X\to Y$ is compact. 

\item[(ii)] Let $X$ be a vector lattice and $Y$ be a normed space. An operator $T\in L(X,Y)$ is said to be $AM$-compact if $T[-x,x]$ is relatively compact for every $x\in X_+$; see \cite[Def.3.7.1]{M}. Let $(X,\tau)$ be a locally solid vector lattice with order bounded $\tau$-neighborhood and $(Y,\lVert\cdot\rVert_Y)$ be a normed vector lattice. Then $T\in L(X,Y)$ is $AM$-compact operator iff $T:(X,\lvert\cdot\rvert,X_\tau)\to(Y,\lVert\cdot\rVert_Y,\mathbb{R})$ is $p_\tau$-compact; apply \cite[Thm.2.2]{L} and \cite[Thm.2.19(i)]{AB}.
\end{enumerate}
\end{exam}

\begin{lem}\label{sum is also ptcontinuous}
If $S$ and $T$ are (sequentially) $p_\tau$-compact operators between $LSNVL$s then $T+S$ and $\lambda T$, for any real number $\lambda$, are also (sequentially) $p_\tau$-compact operators.
\end{lem}

\begin{prop}\label{leftandrightmultiplication}
Let $(X,p,E_\tau)$ be an $LSNVL$ and $R,T,S \in L(X)$. 
\begin{enumerate}
\item[(i)] If $T$ is a $($sequentially$)$ $p_\tau$-compact and $S$ is a (sequentially) $p_\tau$-continuous operators then $S\circ T$ is (sequentially) $p_\tau$-compact.
\item[(ii)] If $T$ is a $($sequentially$)$ $p_\tau$-compact and $R$ is a $p_\tau$-bounded operators then $T\circ R$ is (sequentially) $p_\tau$-compact.
\end{enumerate}
\end{prop}

\begin{rem}\
\begin{enumerate}
\item[(i)] Let $X$ be an $LSNVL$ and $(Y,\acute{\tau})$ be a locally solid vector lattice with $Y$ being compact. Then each operator $T:(X,p,E_\tau)\to (Y,\lvert\cdot\rvert,Y_{\acute{\tau}})$ is (sequentially) $p_\tau$-compact.
		
\item[(ii)] Let $X$ be an $LSNVL$ and $(Y,\lVert\cdot\rVert_Y)$ be a finite dimensional normed space, and $\acute{\tau}$ be the topology generated by this norm. If $T:(X,p,E_\tau)\to (Y,\lvert\cdot\rvert,Y_{\acute{\tau}})$ is $p_\tau$-bounded operator then it is sequentially $p_\tau$-compact.

\item[(iii)] Let $(X,\tau)$ be a locally solid vector lattice with an order bounded $\tau$-neighborhood	of zero and $(Y,m,F_{\acute{\tau}})$ be an $op_\tau$-continuous $LSNVL$ with $Y$ being an atomic $KB$-space. If $T:X\to Y$ is order bounded operator then $T:(X,\lvert\cdot\rvert,X_\tau)\to(Y,m,F_{\acute{\tau}})$ is $p_\tau$-compact; see \cite[Thm.2.2]{L} and \cite[Rem.6]{AEEM2}.
\end{enumerate}
\end{rem}

\begin{ques}
Is it true that a $p_\tau$-compact operator is $p_\tau$-bounded?	
\end{ques}

\begin{rem}\label{there is subsequence}
Let $(T_m)$ be a sequence of sequentially $p_\tau$-compact operators from $X$ to $Y$. For a given $p_\tau$-bounded sequence $(x_n)$ in $X$, by a standard diagonal argument, there exists a subsequence $(x_{n_k})$ such that, for any $m\in \mathbb{N}$, $T_mx_{n_k}\ptc y_m$ for some $y_m\in Y$.
\end{rem}

\begin{thm}\label{order conv is compact}
Let $(T_m)$ be a sequence of order bounded sequentially $p_\tau$-compact operators from $(X,p,E_\tau)$ to a sequentially $p_\tau$-complete $op_\tau$-continuous $(Y,q,F_{\acute{\tau}})$ with $Y$ being order complete. If $T_m\oc T$ in $L_b(X,Y)$ then $T$ is sequentially $p_\tau$-compact.
\end{thm}

\begin{proof}
Let $(x_n)$ be a $p_\tau$-bounded sequence in $X$. By Remark \ref{there is subsequence}, there exists a subsequence $(x_{n_k})$ such that, for any $m\in \mathbb{N}$, $T_mx_{n_k}\ptc y_m$ for some $y_m\in Y$. We show that $(y_m)$ is a $p_\tau$-Cauchy sequence. Consider the following formula
\begin{eqnarray}
q(y_m-y_j)\leq q(y_m-T_mx_{n_k})+q(T_mx_{n_k}-T_jx_{n_k})+q(T_jx_{n_k}-y_j). 
\end{eqnarray}
The first and the third terms in the last inequality both $\acute{\tau}$-converge to zero as $m\to\infty$ and $j\to\infty$, respectively. Since $T_m\oc T$, we have $T_mx_{n_k}\oc Tx_{n_k}$ for all $x_{n_k}$; see \cite[Thm.VIII.2.3]{V}. Then, for a fixed index $k$, we have
$$
\lvert T_mx_{n_k}-T_jx_{n_k}\rvert\leq \lvert T_mx_{n_k}-Tx_{n_k}\rvert+\lvert Tx_{n_k}-T_jx_{n_k}\rvert\oc 0
$$
as $m,j\to\infty$, and so $(T_m-T_j)x_{n_k}\oc 0$ in $Y$. Hence, by $op_\tau$-continuity of $(Y,q,F_{\acute{\tau}})$, we get $q(T_mx_{n_k}-T_jx_{n_k})\tcc 0$ in $F$. By the formula $(1)$, $(y_m)$ is $p_\tau$-Cauchy. Since $Y$ is sequentially $p_\tau$-complete, there is $y\in Y$ such that $q(y_m-y)\tcc 0$ in $F$ as $m\to\infty$. So, for arbitrary $m$, if we take $\acute{\tau}$-limit with $k$ in the following formula
$$
q(Tx_{n_k}-y)\leq q(Tx_{n_k}-T_mx_{n_k})+q(T_mx_{n_k}-y_m)+q(y_m-y),
$$
we get $\acute{\tau}-\lim q(Tx_{n_k}-y)\leq q(Tx_{n_k}-T_mx_{n_k})+q(y_m-y)$ because $q(T_mx_{n_k}-y_m)\tcc 0$. Since $m$ is arbitrary, $\acute{\tau}-\lim q(Tx_{n_k}-y)\tcc 0$. Therefore, $T$ is sequentially $p_\tau$-compact.
\end{proof}

Similar to Theorem \ref{order conv is compact}, we give the following theorem by using equicontinuously and uniformly convergence.
\begin{thm}\label{equicontinuous convergence implies p compactness}
Let $(T_m)$ be a sequence of sequentially $p_\tau$-compact operators from $(X,\lvert\cdot\rvert,X_\tau)$ to a sequentially $p_\tau$-complete $LSNVL$ $(Y,\lvert\cdot\rvert,Y_\tau)$. Then the followings hold;
\begin{enumerate}
\item[(i)] If $(T_m)$ converges equicontinuously to an operator $T:(X,\lvert\cdot\rvert,X_\tau)\to (Y,\lvert\cdot\rvert,Y_{\acute{\tau}})$ then $T$ is sequentially $p_\tau$-compact,

\item[(ii)] If $(T_m)$ uniformly converges on zero
neighborhoods to an operator $T:(X,\lvert\cdot\rvert,X_\tau)\to (Y,\lvert\cdot\rvert,Y_{\acute{\tau}})$ then $T$ is sequentially $p_\tau$-compact.
\end{enumerate}
\end{thm}

\begin{ques}
Is it true that the modulus of (sequentially) $p_\tau$-compact operator is (sequentially) $p_\tau$-compact.
\end{ques}

Let $(X,E)$ be a decomposable $LNS$ and $(Y,F)$ be an $LNS$ with $F$ being order complete. Then each dominated operator $T:X\to Y$ has the exact dominant $\pmb{\lvert} T \pmb{\rvert}:E\to F$; see \cite[4.1.2,p.142]{K}. For a sequence $(T_n)$ in the set of dominated operators $M(X,Y)$, we call $T_n\to T$ in $M(X,Y)$ whenever $\pmb{\lvert} T_n-T\pmb{\rvert}(e)\tcc 0$ in $F$ for each $e\in E$. 
\begin{thm}\label{conv in dominated operator}
Let $(X,p,E_\tau)$ be a decomposable and $(Y,q,F_{\acute{\tau}})$ be a sequentially $p_\tau$-complete $LSNVL$s with $F$ being order complete. If $(T_m)$ is a sequence of sequentially $p_\tau$-compact operators and $T_m\to T$ in $M(X,Y)$ then $T$ is sequentially $p_\tau$-compact. 
\end{thm}

\begin{proof}
Let $(x_n)$ be a $p_\tau$-bounded sequence in $X$. By Remark \ref{there is subsequence}, there exists a subsequence $(x_{n_k})$ and a sequence $(y_m)$ in $Y$ such that, for any $m\in \mathbb{N}$, $T_mx_{n_k}\ptc y_m$. We show that $(y_m)$ is $p_\tau$-Cauchy sequence in $Y$. Consider the formula $(1)$ of Theorem \ref{order conv is compact}. Similarly, the first and the third terms in the last inequality of $(1)$ both $\acute{\tau}$-converge to zero as $m\to\infty$ and $j\to\infty$, respectively. Since $T_m\in M(X,Y)$ for all $m\in \mathbb{N}$,
$$
q(T_mx_{n_k}-T_jx_{n_k})\leq \pmb{\lvert} T_m-T_j\pmb{\rvert}\big(p(x_{n_k})\big)\leq\pmb{\lvert} T_m-T\pmb{\rvert}\big(p(x_{n_k})\big)+ \pmb{\lvert} T-T_j\pmb{\rvert}\big(p(x_{n_k})\big)\tcc 0
$$
as $m,j\to\infty$. Thus, $q(y_m-y_j)\tcc 0$ in $F$ as $m,j\to\infty$. Therefore, $(y_m)$ is $p_\tau$-Cauchy. Since $Y$ is sequentially $p_\tau$-complete, there is $y\in Y$ such that $q(y_m-y)\tcc 0$ in $F$ as $m\to\infty$. By the following formula
\begin{eqnarray*}
q(Tx_{n_k}-y)&\leq& q(Tx_{n_k}-T_mx_{n_k})+q(T_mx_{n_k}-y_m)+q(y_m-y)\\ &\leq& \pmb{\lvert} T_m-T 
\pmb{\rvert}\big(p(x_{n_k})\big)+q(T_mx_{n_k}-y_m)+q(y_m-y)
\end{eqnarray*}
and by repeating the same of last part of Theorem \ref{order conv is compact}, we get $q(Tx_{n_k}-y)\tcc 0$. Therefore, $T$ is sequentially $p_\tau$-compact.
\end{proof}

\begin{prop}\label{seq.p-compactimpliescompact}
Let $(X,p,E_\tau)$ be an $LSNVL$, where $(E,\lVert\cdot\rVert_E)$ is an $AM$-space with a strong unit, and $(Y,m,F_{\acute{\tau}})$ be an $LSNVL$, where $(F,\lVert\cdot\rVert_F)$ is normed vector lattice and $\acute{\tau}$ is generated by the norm $\lVert\cdot\rVert_F$. If $T:(X,p,E_\tau)\to(Y,m,F_{\acute{\tau}})$ is sequentially $p_\tau$-compact then $T:(X,p\text{-}\lVert\cdot\rVert_E)\to(Y,m\text{-}\lVert\cdot\rVert_F)$ is compact.
\end{prop}

\begin{proof}
Let $(x_n)$ be a normed bounded sequence in $(X,p\text{-}\lVert\cdot\rVert_E)$. That is $p\text{-}\lVert x_n\rVert_E=\lVert p(x_n)\rVert_E<\infty$ for all $n\in N$. Since $(E,\lVert\cdot\rVert_E)$ is an $AM$-space with a strong unit, $p(x_n)$ is order bounded in $E$. Thus, $p(x_n)$ is $\tau$-bounded in $E$; see \cite[Thm.2.19(i)]{AB}. So, $(x_n)$ is a $p_\tau$-bounded sequence in  $(X,p,E_\tau)$. Since $T$ is sequentially $p_\tau$-compact, there are a subsequence $x_{n_{k}}$ and $y\in Y$ such that $m(Tx_{n_{k}}-y)\tcc 0$ in $F$. Then $\lVert m(Tx_{n_{k}}-y)\rVert_F\to 0$ or $m\text{-}\lVert Tx_{n_{k}}-y\rVert_F\to 0$ in $F$. Thus, the operator $T:(X,p\text{-}\lVert\cdot\rVert_E)\to(Y,m\text{-}\lVert\cdot\rVert_F)$ is compact.
\end{proof}

It is known that finite rank operator is compact. Similarly, we see the following work.
\begin{prop}\label{rankoneoperator}
Let $(X,p,E_\tau)$ and $(Y,m,F_{\acute{\tau}})$ be $LSNVL$s with $(F,\acute{\tau})$ having the Lebesgue property. Consider an operator $T:(X,p,E_\tau)\to(Y,m,F_{\acute{\tau}})$ defined by $Tx=f(x)y_0$, where $y_0\in Y$ and $f$ is a linear functional on $X$. If $f:(X,p,E_\tau)\to (\mathbb{R},\lvert \cdot \rvert,\mathbb{R})$ is $p_\tau$-bounded then $T$ is (sequentially) $p_\tau$-compact.
\end{prop}

\begin{proof}
Suppose $(x_\alpha)$ is a $p_\tau$-bounded net in $X$. Since $f$ is $p_\tau$-bounded, $f(x_\alpha)$ is bounded in $\mathbb{R}$. Then there is a subnet $(x_{\alpha_{\beta}})$ such that $f(x_{\alpha_{\beta}}) \to \lambda$ for some $\lambda \in \mathbb{R}$. For $y_0\in Y$, we have the following formula
$$
m(Tx_{\alpha_{\beta}}-\lambda y_0)=m(f(x_{\alpha_{\beta}})y_0-\lambda y_0)=m\big((f(x_{\alpha_{\beta}})-\lambda) y_0\big)=\lvert f(x_{\alpha_{\beta}})-\lambda \rvert m(y_0)\oc 0.
$$
By the Lebesgue property of $F$, we get $m(Tx_{\alpha_{\beta}}-\lambda y)\tc 0$ in $E$. Thus, $T$ is $p_\tau$-compact.
\end{proof}

\begin{prop}\label{dominated operatpr is ptau compact}
Let $(X,p,E_\tau)$ be an $LSNVL$ with $(E,\tau)$ having an order bounded $\tau$-neighborhood and $(Y,m,F_{\acute{\tau}})$ be an $LSNVL$, where $(Y,\lVert\cdot\rVert_Y)$ is an order continuous atomic $KB$-space and $\acute{\tau}$ is generated by $\lVert\cdot\rVert_Y$. If $T:(X,p,E_\tau)\to(Y,\lvert\cdot\rvert,Y_{\acute{\tau}})$ is $p$-bounded or dominated operator then it is $p_\tau$-compact.
\end{prop}

Recall that a linear operator $T$ from an LNS $(X,E)$ to a Banach space $(Y,\lVert\cdot\rVert_Y)$ is called \textit{generalized $AM$-compact} or \textit{$GAM$-compact} if, for any $p$-bounded set $A$ in $X$, $T(A)$ is relatively compact in $(Y,\lVert\cdot\rVert_Y)$. 
\begin{prop}\label{sequentially GAM-compact}
Let $(X,p,E_\tau)$ be an $LSNVL$ with $(E,\tau)$ having an order bounded $\tau$-neighborhood and $(Y,m,F_{\acute{\tau}})$ be an $op_\tau$-continuous $LSNVL$ with a Banach lattice $(Y,\lVert\cdot\rVert_Y)$. If $T:(X,p,E_\tau)\to(Y,\lVert\cdot\rVert_Y)$ is $GAM$-compact then $T:(X,p,E_\tau)\to(Y,m,F_{\acute{\tau}})$ is sequentially $p_\tau$-compact.
\end{prop}

\begin{proof}
Let $(x_n)$ be a $p_\tau$-bounded sequence in $X$. By \cite[Thm.2.2]{L}, $(x_n)$ is $p$-bounded in $(X,p,E_\tau)$. Since $T$ is $GAM$-compact, there are a subsequence $(x_{n_k})$ and some $y\in Y$ such that  $\lVert Tx_{n_k}-y\rVert_Y\to 0$. Since $(Y,\lVert\cdot\rVert_Y)$ is Banach lattice then, by \cite[Thm.VII.2.1]{V}, there is a further subsequence $(x_{n_{k_j}})$ such that $Tx_{n_{k_j}}\oc y$ in $Y$. Then, by $op_\tau$-continuity of  $(Y,m,F_{\acute{\tau}})$, we get $Tx_{n_{k_j}}\ptc y$ in $Y$. Hence, $T$ is sequentially $p_\tau$-compact.
\end{proof}

\begin{ques}
Recall that a norm bounded operator between Banach spaces is compact iff its adjoint is likewise compact. Similarly, is it true that adjoint of $p_\tau$-compact operator is $p_\tau$-compact?
\end{ques}

\begin{prop}
Let $(X,\lVert\cdot\rVert_X)$ be a normed lattice and $(Y,\lVert\cdot\rVert_Y)$ be a Banach lattice. If $T:(X,\lVert\cdot\rVert_X,\mathbb{R})\to(Y,\lvert\cdot\rvert,Y_{\acute{\tau}})$ is sequentially $p_\tau$-compact and $p$-bounded, and $f:Y\to\mathbb{R}$ is $\sigma$-order continuous then $(f\circ T):X\to \mathbb{R}$ is compact.
\end{prop}

\begin{proof}
Assume $(x_n)$ be a norm bounded sequence in $X$. Since $T$ is sequentially $p_\tau$-compact, there are a subsequence $(x_{n_k})$ and $y\in Y$ such that $Tx_{n_k}\ptc y$ or $\lvert Tx_{n_k}-y\rvert\tcc 0$ or $Tx_{n_k} \xrightarrow{\lVert\cdot\rVert_Y} y$ in $Y$. Since $(Y,\lVert\cdot\rVert_Y)$ be Banach lattice, there is a further subsequence $(x_{n_{k_j}})$ such that $Tx_{n_{k_j}}\oc y$ in $Y$; see \cite[Thm.VII.2.1]{V}. By $\sigma$-order continuity of $f$, we have $(f\circ T)x_{n_{k_j}}\to f(y)$ in $\mathbb{R}$. 
\end{proof}

We now turn our attention to the $up_\tau$-compact operators.
\begin{defn}
Let $X$ and $Y$ be two $LSNVL$s and $T\in L(X,Y)$. Then $T$ is called {\em $up_\tau$-compact} if, for any $p_\tau$-bounded net $(x_\alpha)$ in $X$, there is a subnet $(x_{\alpha_\beta})$ such that $Tx_{\alpha_\beta}\uptc y$ in $Y$ for some $y\in Y$. If the condition holds only for sequences then $T$ is called {\em sequentially-$up_\tau$-compact}. 
\end{defn}

It is clear that a $p_\tau$-compact operator is $up_\tau$-compact, and similar to Lemma \ref{sum is also ptcontinuous} linear properties hold for $up_\tau$-compact operators. Moreover, an operator $T\in L (X,Y)$ is $($sequentially$)$ $un$-compact iff $T:(X,\lVert\cdot\rVert_X,\mathbb{R})$ $\to(Y,\lVert\cdot\rVert_Y,\mathbb{R})$ is $($sequentially$)$ $up_\tau$-compact; see \cite[Sec.9,p.28]{KMT}. Similar to Proposition \ref{leftandrightmultiplication}, we give the following results. 
\begin{prop}
Let $(X,p,E_\tau)$ be an $LSNVL$ and $R,T,S,H \in L(X)$.
\begin{enumerate}
	\item[(i)] If $T$ is an $($sequentially$)$ $up_\tau$-compact and $S$ is a (sequentially) $p_\tau$-continuous then $S\circ T$ is (sequentially) $up_\tau$-compact.
	
	\item[(ii)] If $T$ is an $($sequentially$)$ $up_\tau$-compact and $R$ is a $p_\tau$-bounded then $T\circ R$ is (sequentially) $up_\tau$-compact.
\end{enumerate}
\end{prop}

Now, we investigate a relation between sequentially $up_\tau$-compact operators and dominated lattice homomorphisms. 
The following is a more general extension of \cite[Prop.9.4]{KMT} and \cite[Thm.8]{AEEM2}, and its proof is similar to Theorem \ref{sequentially $up$-continuous}.

\begin{thm}\label{sequentially $up$-compact}
Let $(X,p,E_\tau)$, $(Y,m,F_{\acute{\tau}})$ and $(Z,q,G_{\hat{\tau}})$ be $LSNVL$s with $(F,\lVert\cdot\rVert_F)$ being Banach lattice and $(G,\lVert\cdot\rVert_G)$ normed lattice, and $\acute{\tau}$ and $\hat{\tau}$ are being generated by the norms. Then the followings hold;
\begin{enumerate}
\item[(i)] If $T\in L(X,Y)$ is a sequentially $up_\tau$-compact operator and $S\in L(Y,Z)$ is a dominated surjective lattice homomorphism then $S\circ T$ is sequentially $up_\tau$-compact;
		
\item[(ii)] If $T\in L(X,Y)$ is a sequentially $up_\tau$-compact, and $S\in L(Y,Z)$ is a dominated lattice homomorphism and $S(Y)$ is $up_\tau$-regular in $Z$ then $S\circ T$ is sequentially $up_\tau$-compact;
		
\item[(iii)] If $T\in L(X,Y)$ is a sequentially $up_\tau$-compact, and $S\in L(Y,Z)$ is a dominated lattice homomorphism operator and $I_{S(Y)}$ $($the ideal generated by $S(Y)$$)$ is $up_\tau$-regular in $Z$ then $S\circ T$ is sequentially $up_\tau$-compact.
\end{enumerate}
\end{thm}

\begin{prop}\label{positive and dominated is ptau compact}
Let $(X,p,E_\tau)$ be an $LSNVL$ and $(Y,m,F_{\acute{\tau}})$ be an $up_\tau$-complete $LSNVL$, and $S,\ T:(X,p,E_\tau)\to(Y,m,F_{\acute{\tau}})$ be operators with $0\leq S\leq T$. If $T$ is a lattice homomorphism and (sequentially) $up_\tau$-compact then $S$ is (sequentially) $up_\tau$-compact.
\end{prop}
\begin{proof}
We will prove the sequential case; the other case is similar. Let $(x_n)$ be a $p_\tau$-bounded sequence in $X$. So, there are a subsequence $(x_{n_k})$ and some $y\in Y$ such that $Tx_{n_k}\uptc y$ in $Y$. In particular, it is $up_\tau$-Cauchy. Fix $u\in Y_+$ and note that
$$
\lvert Sx_{n_k}-Sx_{n_j}\rvert\wedge u\leq (S\lvert x_{n_k}-x_{n_j}\rvert)\wedge u\leq(T\lvert x_{n_k}-x_{n_j}\rvert)\wedge u=\lvert Tx_{n_k}-Tx_{n_j}\rvert\wedge u\tcc0
$$
as $k,j\to \infty$. Thus, we get $(Sx_{n_k})$ is a $up_\tau$-Cauchy sequence in $Y$. Therefore, it follows from $up_\tau$-complete of $Y$.
\end{proof} 

\begin{lem}\label{extension is upt continuous}
Let $(X,p,E_\tau)$ and $(Y,m,F_{\acute{\tau}})$ be two $LSNVL$s with $Y$ being order complete vector lattice. If $T:(X,p,E_\tau)\to (Y,m,F_{\acute{\tau}})$ is a positive $up_\tau$-compact operator then the operator $S:(X_+,p,E_\tau)\to (Y_+,m,F_{\acute{\tau}})$ defined by $S(x)=\sup\{T(u\wedge x):u\in X_+\}$ for each $x\in X_+$ is also $up_\tau$-compact operator.
\end{lem}

\begin{proof}
Suppose $(y_\beta)$ is a $p_\tau$-bounded net in $X_+$. Then there is a subnet $(y_{\beta_\gamma})$ such that $Ty_{\beta_\gamma}\uptc y$ for some $y\in Y$, and so $m\big(\lvert Ty_{\beta_\gamma}-y\rvert\wedge w\big)\tcc 0$ in $F$ for all $w\in Y_+$. For $u\in X_+$ and fixed $w\in Y_+$, we have
$0\leq T(u\wedge y_{\beta_\gamma})\leq T(y_{\beta_\gamma})$, and so $\lvert T(u\wedge y_{\beta_\gamma})-y\rvert\wedge w\leq\lvert T(y_{\beta_\gamma})-y\rvert\wedge w$. By taking supremum over $u\in X_+$, we get $\lvert Sy_{\beta_\gamma}-y\rvert\wedge w\leq\lvert T(y_{\beta_\gamma})-y\rvert\wedge w\tcc 0$, and so we get the desired result.
\end{proof}

\begin{rem}\label{sum of pbounded is p bounded}
The sum of two $p_\tau$-bounded subsets is also $p_\tau$-bounded since the sum of two solid subsets is solid. Moreover, for a $p_\tau$-bounded net $(x_\alpha)$ in an $LSNVL$ $(X,p,E_\tau)$, the nets $(x^+_\alpha)$ and $(x^-_\alpha)$ are $p_\tau$-bounded. 
\end{rem}
The following theorem is $up_\tau$-compact version of Proposition \ref{S is uptau contınuous}, so we omit its proof.
\begin{thm}\label{S is uptau contınuous}
Let $(X,p,E_\tau)$ and $(Y,m,F_{\acute{\tau}})$ be two $LSNVL$s with $Y$ being order complete vector lattice. If $T:(X,p,E_\tau)\to (Y,m,F_{\acute{\tau}})$ is a positive $up_\tau$-compact operator then the Kantorovich extension of $S:(X_+,p,E_\tau)\to (Y_+,m,F_{\acute{\tau}})$ defined by $S(x)=\sup\{T(x_\alpha\wedge x):x_\alpha\in X_+\ \text{is} \ p_\tau\text{-bounded}\}$ for each $x\in X_+$ is also $up_\tau$-compact.
\end{thm}
%%%%%%%%%%%%%%%%%%%%%%%%%%%%%%%%%%%%%%%%%%%%%%%%%%%%%%%%%%%%%%%%%%%%%%%%%%

\end{document}